\documentclass[leqno, 11pts]{article}

\usepackage{graphicx}

\usepackage{stmaryrd}
\usepackage{amsmath,amssymb,amsfonts}
\usepackage{times}
\usepackage[T1]{fontenc}
\usepackage{graphicx}
\usepackage[english]{babel}
\usepackage{amsmath}
\usepackage{amsthm}
\usepackage{amssymb}
\usepackage{enumerate}
\usepackage{xspace}
\usepackage{euscript}
\usepackage{graphicx}
\usepackage{amscd}
\usepackage{epsfig}
\usepackage{tabularx}
\usepackage{color}
\usepackage{bm}
\usepackage{tikz} 

\newcommand{\mv}[1]{\left\{\!\!\left\{#1\right\}\!\!\right\}}

\newtheorem{theorem}{Theorem}[section]
\newtheorem{lemma}[theorem]{Lemma}

\theoremstyle{definition}
\newtheorem{definition}[theorem]{Definition}

\theoremstyle{remark}
\newtheorem{remark}[theorem]{Remark}

\numberwithin{equation}{section}

\begin{document}

\title{Regular Convergence and Finite Element Methods for Eigenvalue Problems}



	\author{
		Bo Gong\thanks{Department of Mathematical Sciences, Beijing University of Technology, Beijing 100124, People?s Republic of China. {\it gongbo@bjut.edu.cn}}
		\and
		Jiguang Sun\thanks{Department of Mathematical Sciences, Michigan Technological University, Houghton, MI 49931, U.S.A. {\it jiguangs@mtu.edu}}
	}
\date{}

\maketitle

\begin{abstract}
Regular convergence, together with various other types of convergence, has been studied since the 1970s for the discrete approximations of linear operators. In this paper, we consider the eigenvalue approximation of compact operators whose spectral problem can be written as the eigenvalue problem of some holomophic Fredholm operator function. Focusing on the finite element methods (conforming, discontinuous Galerkin, etc.), we show that the regular convergence of discrete holomorphic operator functions follows from the approximation property of the finite element spaces and the compact convergence of the discrete operators in some suitable Sobolev space. The convergence for eigenvalues is then obtained using the discrete approximation theory for the eigenvalue problems of holomorphic Fredholm operator functions. The result can be used to show the convergence of various finite element methods for eigenvalue problems such as the Dirhcilet eigenvalue problem and the biharmonic eigenvalue problem.
\end{abstract}

\section{Introduction}
Eigenvalue problems of partial differential equations have many important applications in science and engineering, e.g., 
design of solar cells for clean energy, calculation of electronic structure in condensed matter, extraordinary optical transmission, non-destructive testing, photonic crystals,
and biological sensing.
Due to the flexibility in treating complex structures and rigorous theoretical justification,
finite element methods have been widely used to compute eigenvalue problems
\cite{BrambleOsborn1973MC, Osborn1975MC,DesclouxNassifRappaz1978RAIROA, BabuskaOsborn1991, Boffi2010AN, SunZhou2016}.

The study of the finite element methods for eigenvalue problems started in 1970s
and has been an active research area since then. Many results obtained before 1990s can be found in the book chapter by Babu\v{s}ka and Osborn \cite{BabuskaOsborn1991} (see also \cite{Boffi2010AN, SunZhou2016} for some recent developments). 
The main functional analysis tool is the spectral perturbation theory for linear compact operators \cite{Kato1966, Chatelin1983}. 
Essentially, if the uniform convergence of the finite element solution operators to the continuous compact operator, the theory of Babu\v{s}ka and Osborn \cite{BabuskaOsborn1991} 
can be employed to obtain the convergence for the eigenvalues and the associated eigenfunctions, i.e., all eigenpairs are approximated and there are no spurious modes. 

While the uniform convergence of the discrete operators can be proved for some finite element methods such as the conforming finite element methods, it is impossible or
challenging for many other finite element methods, e.g., the discontinuous Galerkin methods. Various methods has been proposed in literature to prove the convergence
when the uniform convergence is not available \cite{Mercier1981MC, Antonietti2006CMAME, DesclouxNassifRappaz1978RAIROA}. In this paper, 
we consider the eigenvalue approximation of compact operators whose spectral problem and reformulate them as  the eigenvalue problems 
of some holomophic Fredholm operator functions.
The regular convergence of discrete holomorphic operator functions
follows from the approximation property of the finite element spaces and the compact convergence of the discrete operators in some suitable Sobolev space. 
The convergence for exact eigenvalues and eigenfunctions is then obtained using the discrete 
approximation theory for the eigenvalue problems of holomorphic Fredholm operator functions in \cite{Karma1996a, Karma1996b}. 
This work extends the study in \cite{GrigorieffJeggle1973, Vainikko1976} and provides an alter way to 
analyze the convergence of various finite element approximations for eigenvalue problems. 

The rest of the paper is arranged as follows. In Section~2, we recall the discrete approximation scheme, different types of convergence for linear operators, 
and the abstract approximation theory for the eigenvalue problems of holomorphic Fredholm operator functions. Section~3 contains the study of regular convergence related to the
finite element approximation operators for the partial differential equations in $L^2$ space. It turns out that the approximation property of the finite element space
and the compact convergence of the discrete solution operator guarantee the regular convergence. Using the abstract approximation theory
for holomorphic Fredholm operator functions, we show the convergence of various finite element methods for the Dirichlet eigenvalue problem in Section~4 and
the biharmonic eigenvalue problem in Section~5. We end the paper with some conclusions and future work in Section~6.

\section{Preliminaries}
We introduce the discrete approximation scheme, different types of convergence, and the abstract approximation theory for the eigenvalue problems of
holomorphic Fredholm operator functions. We refer the readers to \cite{Vainikko1976, Stummel1970, Stummel1971, Chatelin1983, Karma1996a, Karma1996b} for more details.

\subsection{Discrete Approximation Scheme}
Let $X$ be a Banach space and $X_n (n \in \mathbb N)$ be a sequence of approximation spaces for $X$. Let $P=\{p_n\}_{n \in \mathbb N}$ be
a sequence of bounded linear operators $p_n: X \to X_n$ such that
\begin{equation} \label{pnxEn}
\|p_nx\|_{X_n} \longrightarrow \|x\|_X \quad (n \in \mathbb N).
\end{equation}

\begin{definition}
A sequence $\{x_n\}_{n \in \mathbb N}$ with $x_n \in X_n$ is said to $P$-converge to $x \in X$ if
\[
\|x_n - p_n x\|_{X_n} \longrightarrow 0 \quad (n \in \mathbb N).
\]
We write it as $x_n \overset{P}{\longrightarrow} x$ or simply $x_n \to x \,(n \in \mathbb N)$.
\end{definition}
If $\{x^{(n)}\} (n \in \mathbb N)$ converges to $x$ in $X$, it holds that $p_n x^{(n)}  \overset{P}{\longrightarrow} x$ (see Chp.~1 of \cite{Vainikko1976}).

\begin{definition}
A sequence $\{x_n\}_{n \in \mathbb N}$ with $x_n \in X_n$ is called (discrete) $P$-compact if for every $\mathbb N' \subset \mathbb N$ there exists $\mathbb N'' \subset \mathbb N'$ and
an $x \in X$ such that $x_n \overset{P}{\longrightarrow} x \,(n \in \mathbb N'')$.
\end{definition}

Let $Y$ be a Banach space. Denote by ${\mathcal L}(X, Y)$ the space of bounded linear operators from $X$ to $Y$.
We denote by $\mathcal{N}(A)=\{x \in X: Ax = 0\}$ and $\mathcal{R}(A)=\{y \in Y: y = Ax, x\in X\}$ the null space and the range
of the operator $A \in \mathcal{L}(X, Y)$, respectively.

\begin{definition}
An operator $A \in \mathcal{L}(X, Y)$ is called semi-Fredholm if $\mathcal{R}(A) \subset Y$ is closed and additionally $\mathcal{N}(A)$ has a finite dimension or 
$\mathcal{R}(A)$ has a finite codimension.
If $\mathcal{R}(A)$ is closed and in addition dim$\mathcal{N}(A)$ and codim$\mathcal{R}(A)$ are both finite, $A$ is called Fredholm.
The index of a Fredholm operator is defined as
\[
\text{ind} A = \text{dim} \mathcal{N}(A) - \text{codim} \mathcal{R}(A).
\]
\end{definition}

In the following, we define various notions of convergence for a linear operator \cite{Stummel1970, Vainikko1976, Chatelin1983}. They are point convergence, stable convergence, compact convergence and regular convergence.
Let $Y_n$ be a sequence of approximation spaces for $Y$ and $Q=\{q_n\}_{n \in \mathbb N}$ be
a sequence of linear operators $q_n: Y \to Y_n$ such that
$\|q_n y\|_{Y_n} \longrightarrow \|y\|_Y\, (n \in \mathbb N)$.
\begin{definition}
A sequence $\{A_n\}$ of linear operators $A_n \in \mathcal{L}(X_n, Y_n)$ converges (or PQ-converges, or converges discretely) to $A \in \mathcal{L}(X, Y)$ 
if the following relation holds for every P-converging sequence $\{x_n\}$:
\[
 x_n  \overset{P}{\longrightarrow}  x \, (n \in \mathbb N) \Longrightarrow A_n x_n   \overset{Q}{\longrightarrow} Ax \, (n \in \mathbb N).
\]
We write it as $A_n \longrightarrow A\, (n \in \mathbb N)$ or $A_n  \overset{PQ }{\longrightarrow} A \,(n \in \mathbb N)$.
\end{definition}

Theorem 2.8 in \cite{Vainikko1976}) claims the equivalence of $A_n  \overset{PQ }{\longrightarrow}  A \, (n \in \mathbb N)$ and
\[
\|A_n\| \le const \,(n \in \mathbb N)\quad \text{and} \quad \|A_np_n x - q_n A x\| \to 0\, (n \in \mathbb N) \,\, \forall x \in X.
\]
We shall present a theorem for later use. It is a consequence of Theorem 2.8 and an exercise (Theorem 2-9) in \cite{Vainikko1976}.
\begin{theorem}\label{UniformBoundedness}
If $p_n \in \mathcal{L}(E, E_n)$, $p_n E = E_n$ and $b=const.$ with
\begin{equation}\label{210}
\inf_{x \in E, p_n x = x_n} \|x\|_E \le b \|x_n\|_{E_n} \quad (n \in \mathbb N, \forall x_n \in E_n)
\end{equation}
exists, then $\|A_n p_n x - q_n Ax\| \to 0 \, (n \in \mathbb N) \, \forall x \in X$ implies $\|A_n\| \le const \,(n \in \mathbb N)$.
\end{theorem}
\begin{proof}
If the conditions of the above theorem are satisfied, by Theorem 2-9 in \cite{Vainikko1976}, one has the equivalence of $\|A_n p_n x - q_n Ax\| \to 0 \, (n \in \mathbb N) \, \forall x \in X$
and $A_n  \overset{PQ }{\longrightarrow}  A \, (n \in \mathbb N)$. By Theorem 2.8 in \cite{Vainikko1976}, $\|A_n\| \le const \,(n \in \mathbb N)$.
\end{proof}

\begin{definition}
A sequence $\{A_n\}_{n \in \mathbb N}$ of linear operators $A_n \in \mathcal{L}(X_n, Y_n)$ converges stably to $A \in \mathcal{L}(X, Y)$ if the following two conditions are met:
\begin{itemize}
\item[1.] $A_n  \overset{PQ}{\longrightarrow} A \,(n \in \mathbb N)$;
\item[2.] there is some $n_0 \in \mathbb N$ such that the inverse operators $A_n^{-1} \in \mathcal{L}(Y_n, X_n)$ exist for all $n \ge n_0$, where 
	\[
		\|A_n^{-1}\| \le C \quad (n \ge n_0).
	\]
\end{itemize}
We write briefly $A_n \to A$ stably.
\end{definition}
\begin{definition}
A sequence $\{A_n\}_{n \in \mathbb N}$ of linear operators $A_n \in \mathcal{L}(X_n, Y_n)$ converges compactly to $A\in \mathcal{L}(X, Y)$
if the following conditions are met:
\begin{itemize}
\item[1.] $A_n \overset{PQ}{\longrightarrow} A \,(n \in \mathbb N)$;
\item[2.] $ \|x_n\|_{X_n} \le C \,(n\in \mathbb N) \Longrightarrow \{A_n x_n\} \,Q\text{-compact}$.
\end{itemize}
\end{definition}
\begin{definition}
A sequence $\{A_n\}_{n \in \mathbb N}$ of linear operators $A_n \in \mathcal{L}(X_n, Y_n)$ converges regularly to $A\in \mathcal{L}(X, Y)$ if the following two conditions are satisfied:
\begin{itemize}
\item[1.] $A_n \overset{PQ}{\longrightarrow} A \,(n \in \mathbb N)$;
\item[2.] $\|x_n\|_{X_n} \le C, \quad \{A_n x_n\} \,Q\text{-compact} \Rightarrow \{x_n\} \, P\text{-compact}$.
\end{itemize}
We write in short $A_n \to A$ regularly.
\end{definition}

The following theorem from \cite{Vainikko1976} (Theorem~2.55 therein) states a connection among the stable convergence, compact convergence, and regular convergence.
\begin{theorem}\label{SCR}
Let
\begin{eqnarray*}
&& B_n \to B \quad \text{stably} \quad (B_n \in {\mathcal L}(X_n, Y_n), B\in {\mathcal L}(X, Y)),\\
&& C_n \to C \quad \text{compactly} \quad  (C_n \in {\mathcal L}(X_n, Y_n), C\in {\mathcal L}(X, Y)),
\end{eqnarray*}
where $R(B)=F$. Then
\[
A_n := B_n + C_n \longrightarrow B+C =:A \quad \text{regularly}.
\]
\end{theorem}

We also define the (discrete) uniform convergence for late use.
\begin{definition}
A sequence $\{A_n\}_{n \in \mathbb N}$ of linear operators $A_n \in \mathcal{L}(X_n, Y_n)$ converges uniformly to $A\in \mathcal{L}(X, Y)$ if 
$\| A_n p_n - q_n A\| \to 0$ as $n \to \infty$. We write in short $A_n \to A$ uniformly.
\end{definition}

\subsection{Holomorphic Fredholm Operator Functions}
We now introduce the discrete approximation theory for eigenvalue problems of holomorphic Fredholm operator functions \cite{Karma1996a, Karma1996b}.
Let $X$ and $Y$ be complex Banach spaces. 
Let  $\Omega \subseteq \mathbb C$ be a compact simply connected region. 
Let $F: \Omega \to \mathcal{L}(X, Y)$ be a holomorphic operator function on $\Omega$ and, for each $\eta \in \Omega$, 
$F(\eta)$ is a Fredholm operator of index zero \cite{Gohberg2009}. 
\begin{definition}A complex number $\lambda \in \Omega$ is called an eigenvalue
of $F$ if there exists a nontrivial $x \in X$ such that $F(\lambda) x = 0$. The element $x$ is called an eigenelement associated with $\lambda$.
\end{definition}

The resolvent set $\rho(F)$ and the spectrum $\sigma(F) $ of $F(\cdot)$ are defined respectively as
\begin{equation}\label{rhoF}
\rho(F) = \{\eta \in \Omega: F(\eta)^{-1} \text{ exists and is bounded}\}
\end{equation}
and
\begin{equation}\label{sigmaF}
\sigma(F) = \Omega \setminus \rho(F).
\end{equation}
If $\rho(F) \ne \emptyset$, since $F(\eta)$ is holomorphic, the spectrum $\sigma(F)$ has no cluster points in $\Omega$ and every $\lambda \in \sigma(F)$ is
an eigenvalue for $F(\eta)$. Furthermore,  the operator function
$F^{-1}(\cdot)$ is meromorphic (see Section 2.3 of \cite{Karma1996a}).
The dimension of $\mathcal{N}(F(\lambda))$ is called the 
geometric multiplicity of an eigenvalue $\lambda$. 

\begin{definition}
An ordered sequence of elements $x_0, x_1, \ldots, x_k$ in $X$ is called a Jordan
chain of $F$ at an eigenvalue $\lambda$ if
\[
F(\lambda)x_j + \frac{1}{1!} F^{(1)}(\lambda)x_{j-1}+ \ldots + \frac{1}{j!} F^{(j)}(\lambda) x_0 = 0, \quad j = 0, 1, \ldots, k,
\]
where $F^{(j)}$ denotes the $j$th derivative. 
\end{definition}

The length of any Jordan chain of an eigenvalue $\lambda$ is finite.
Denote by $m(F,\lambda, x_0)$ the length of a Jordan chain formed by an eigenelement $x_0$.
The maximal length of all Jordan chains of $\lambda$ is denoted by $\kappa(F, \lambda)$, called the ascent of $\lambda$.
Elements of any Jordan chain of an eigenvalue $\lambda$ are called generalized eigenelements of $\lambda$.
\begin{definition}
The closed linear hull of all generalized eigenelements of an eigenvalue $\lambda$, denoted by $G(\lambda)$,
is called the generalized eigenspace of $\lambda$.
\end{definition}



Let $X_n, Y_n$ be Banach spaces, not necessarily subspaces of $X, Y$. Let $p_n \in \mathcal{L}(X, X_n)$ and
		$q_n \in \mathcal{L}(Y, Y_n)$  be such that
		\begin{equation}\label{pnqn}
			\lim_{n \to \infty} \|p_n x\|_{X_n}=\|x\|_X, \, x \in X \quad \text{and} \quad
			\lim_{n \to \infty} \|q_n y\|_{Y_n}=\|y\|_Y, \, y \in Y.
		\end{equation}
Denote by $\Phi_0(X, Y)$ the sets in $\mathcal{L}(X, Y)$ of all Fredholm operators of index zero.
Consider a sequence of discrete operator functions 
\[
F_n: \Omega \to  \Phi_0(X_n, Y_n), \quad n \in \mathbb N.
\]
Assume that  $\rho(F) \ne \emptyset$ and 
\begin{itemize}
\item[(b1)] $F_n(\eta), \eta \in \Omega,$ is a holomorphic Fredholm operator of index zero;
\item[(b2)] The sequence $\{F_n(\cdot)\}$ is equibounded on $\Omega$, i.e.,
	$\|F_n(\eta)\| \le C, \eta \in \Omega$;
\item[(b3)] $\{F_n(\eta)\}_{n \in N}$ converges to $F(\eta)$ for every $\eta \in \Omega$ and $x \in X$, i.e., 
	\[
	\|[F_n(\eta)p_n-p_n F(\eta)] x\|_{Y_n} \to 0;
	\]	 
\item[(b4)] $F_n(\eta)$ converges regularly to $F(\eta)$ for every $\eta \in \Omega$.
\end{itemize}

The following theorem states that
all eigenvalues are approximated correctly (see \cite{Karma1996a, Karma1996b} or \cite{Beyn2014}). 
\begin{theorem}\label{Thm210} 
Assume that (b1)-(b4) hold. 
 For any $\lambda \in \sigma(F)$ there exists $n_0 \in \mathbb N$ and a sequence $\lambda_n \in \sigma(F_n)$,
	$n \ge n_0$, such that $\lambda_n \to \lambda$ as $n \to \infty$. For any sequence $\lambda_n \in \sigma(F_n)$
	with this convergence property and the associate eigenelements $v_n^0 \in \mathcal{N}(F(\lambda_n)), \|v_n^0\|_{X_n}=1$, 
	one has that
	\begin{eqnarray}
	\label{eigenvalueorder}	|\lambda_n - \lambda| &\le& C \epsilon_n^{1/\kappa}, \\
	\label{eigenvectororder}	\inf_{v \in \mathcal{N}(F(\lambda))} \| v_n^0 - p_n v\|_{X_n} &\le& C \epsilon_n^{1/\kappa},
	\end{eqnarray}
	where
	\[
		\epsilon_n = \max_{|\eta-\lambda| \le \delta}\max_{v \in G(\lambda)} \| F_n(\eta)p_n v - q_n F(\eta) v\|_{X_n}
	\]
	for sufficiently small $\delta > 0$.
\end{theorem}

\section{Finite Element Approximations}\label{FemConv}
To analyze finite element methods for eigenvalue problems of partial differential equations, we choose $X=Y =L^2(D)$ for the rest of the paper, 
where $D \subset \mathbb R^2$ is a Lipschitz polygonal domain. The result in this paper holds for higher dimension cases if the corresponding 
finite element methods approximation
properties for the source problem are available.
Denote by $\|\cdot\|$ the usual $L^2$-norm.
Let $T$ be the solution operator for the source problem associated to the eigenvalue problem. For example, the solution operator for the Poisson equation with homogeneous Dirichlet
boundary condition is associated with the Dirichlet eigenvalue problem.
In this section, we present some general results that can be used to prove the convergence of a large class of finite element methods for eigenvalue problems.
It turns out that one only needs to verify the approximation property of the finite element spaces and
the pointwise convergence of the solution operator for the source problem in $L^2$ norm.

Assume that $T \in \mathcal{L}(X, X)$ is compact and let $I: X \to X$ be the identity operator. 
The eigenvalue problem for $T$ is to find $\lambda \ne 0$ and nontrivial $x \in X$ such that
\begin{equation}\label{Teig}
Tx = \frac{1}{\lambda} x.
\end{equation}

Define $F(\eta): \Omega \to \mathcal{L}(X, X)$ such that
\begin{equation}\label{Feta}
F(\eta) := T - \frac{1}{\eta} I, \quad \eta \in \Omega,
\end{equation} 
where $\Omega$ is a compact subset of $\mathbb C$ such that $0 \notin \Omega$. It is clear that $F(\cdot)$ is a holomorphic Fredholm operator function on $\Omega$.
The eigenvalue problem of $F(\cdot)$ is to find $\lambda \in \Omega$ and $x \in X$ such that
\begin{equation}\label{eigen}
F(\lambda) x = 0.
\end{equation} 
Clearly the above eigenvalue problem is equivalent to the eigenvalue problem \eqref{Teig} for $T$.

Let $\mathcal{T}_n := \mathcal{T}_{h_n}$ be a regular triangular mesh for $D$ with mesh size $h_n \to 0^+, n \to \infty$.
Let $X_n \subset X$ be the associated finite element space endowed with the $L^2$-norm and $P=\{p_n\}$, $p_n: X \to X_n$, is the $L^2$-projection, i.e., for $f \in X$, $p_n f \in X_n$
is such that
\begin{equation}\label{pn}
(p_n f, x_n) = (f, x_n) \quad \text{for all } x_n \in X_n.
\end{equation}

Let $I_n: X_n \to X_n$ be the identity operator. 
Assume that there exists a series of finite element approximation operators
$T_n: X_n \to X_n$ for $T$. Define
the discrete approximation operators
\begin{equation}\label{FnEta}
F_n(\eta) := T_n - \frac{1}{\eta} I_n, \quad \eta \in \Omega.
\end{equation}
If $I_n \to I$ stably and  $T_n \to T$ compactly, then using Theorem~\ref{SCR}, it holds that
\begin{equation}\label{FnR}
F_n(\eta) \to F(\eta) \quad \text{regularly for } \eta \in \Omega.
\end{equation}

For the convergence analysis of the discrete eigenvalues of $F_n(\cdot)$ to those of $F(\cdot)$, we shall see that the following two conditions are sufficient.
\begin{itemize}
\item[(i)] for $x \in X$, $\| p_n x - x \| \to 0$ as $n \to \infty$.
\item[(ii)] $T_n \to T$ compactly.
\end{itemize}
Condition (i) is the approximation property of the finite element spaces $X_n$ and (ii) is the compact convergence of the finite element solution operators to the continuous solution operator  for the source problem.

For the finite element spaces $X_n$ consider in this paper, (i) holds. Then the stable convergence of $I_n$ to $I$ is implied by the discrete convergence of $I_n$ to $I$, i.e., $I_n \to I$. 
\begin{lemma} 
Let Condition (i) hold. Then $I_n \to I$ stably.
\end{lemma}
\begin{proof}
Since (i) holds, $I_n \to I$. Furthermore, $I_n^{-1}: X_n \to X_n$ exists and is bounded, $I_n \to I$ stably.
\end{proof}
 
 The following lemmas are on the sufficient conditions for the compact convergence.
 \begin{lemma}\label{CompactConv}
 Let Condition (i) hold and assume that $T_n \to T$ discretely. If $X_n \subset X'$ such that $X'$ is compactly embedded in $X$, then $T_n \to T$ compactly. 
 \end{lemma}
 \begin{proof}
 Let  $\|x_n\| \le C \,(n\in \mathbb N)$. Then $\{T_n x_n\} \,(n\in \mathbb N)$ is bounded in $X'$. Due to the compact embedding of $X'$ into $X$, there exist a convergent subsequence of $\{T_n x_n\} \,(n\in \mathbb N)$, denoted by $\{T_n x_n\} \,(n\in \mathbb N' \subset \mathbb N)$, such that $T_n x_n \to y \in X, n \in \mathbb N'$. Hence $T_n \to T$ compactly. The proof is complete.
 \end{proof}
 
\begin{remark}\label{Tnprime}
For a finite element approximation of the source problem, one usually has a discrete solution operator $T_n': X \to X_n$.
In general, the discrete operator $T_n: X_n \to X_n$ is such that $T_n p_n =T'_n$. 
\end{remark}

\begin{lemma}\label{U2C}
Let Condition (i) hold and assume that $X_n \subset X, n \in \mathbb N$. Let $T: X \to X$ be a compact operator. If $T_n \to T$ uniformly,
then $T_n \to T$ compactly.
\end{lemma}
\begin{proof}
Let  $\|x_n\| \le C \,(n\in \mathbb N)$. Since $T$ is compact, $\{Tx_n\}$ has a convergent subsequence $\{Tx_n\}_{n \in N'}, N'\subset \mathbb N$ such that
$Tx_n \to y \in X $ as $N' \ni n \to \infty$.
For $n \in N'$, 
\begin{eqnarray*}
\|T_n x_n - p_n y \| &=& \|T_n x_n - p_nTx_n + p_nTx_n -p_n y\| \\
	 &\le& \|T_n p_n x_n - p_nT x_n\| + \|p_nT x_n - p_n y\| \\
	 &\to& 0 \quad \text{as } n \to \infty.
\end{eqnarray*}
Hence $\{T_n x_n\}_{n \in \mathbb N}$ is discretely compact and $T_n \to T$ compactly.
\end{proof}

\begin{theorem}\label{FnFregular}
Let $T: X \to X$ be compact. Assume that $T_n \to T$ compactly and $I_n \to I$ stably. 
Then 
\begin{itemize}
\item[1.] $\|F_n(\eta)\| \le C$ for $\eta \in \Omega$;
\item[2.] For every $\eta \in \Omega$ and $x \in X$, $\|[F_n(\eta)p_n-p_n F(\eta)] x\| \to 0$;
\item[3.] $F_n(\eta) \to F(\eta)$ regularly for each $\eta \in \Omega$.
\end{itemize}
\end{theorem}
\begin{proof}
\begin{itemize}
\item[1.] Since $X_n \subset X$ and $p_n$ are the $L^2$-projections, we clearly have that $p_n \in \mathcal{L}(X, X_n)$ and $p_n X = X_n$.
Furthermore, since $p_n$ are orthogonal projections, \eqref{210} is satisfied by taking $x=x_n$ and $b=1$.
Since $T_n \to T$ discretely, $T_n$ are uniformly bounded in $n$ due to Theorem~\ref{UniformBoundedness}. It is clear that $I_n$ are uniformly bounded. 
Then $\|F_n(\eta)\| \le C$ due to the fact that $\Omega$ is compact.
\item[2.] For a fixed $\eta \in \Omega$ and $x \in X$, 
	\begin{eqnarray*}
		\left\|[F_n(\eta)p_n-p_n F(\eta)] x \right\| &=& \left\|(T_n p_n - p_n T)x - \frac{1}{\eta} (I_n p_n - p_n I) x \right \| \\
		&\le& \|(T_n p_n - p_n T)x \|  + \left\| \frac{1}{\eta} (I_n p_n - p_n I) x \right\|  \\
		&\to& 0 \quad \text{as } n \to \infty.
	\end{eqnarray*}
\item[3.] Since $T_n \to T$ compactly and $I_n \to I $ stably, then $F_n(\eta) \to F(\eta)$ regularly for each $\eta \in \Omega$ due to Theorem~\ref{SCR}.
\end{itemize}
\end{proof}

\begin{theorem}\label{FEMconvergence}
Let $X=L^2(D)$ and $X_n \subset X$ be a sequence of finite element spaces. Let $P=\{p_n\}$ where $p_n: X \to X_n$ is the $L^2$-projection.
Let $T: X \to X$ be a compact operator and $T_n: X_n \to X_n$ be a sequence of discrete operators. 
Assume that conditions (i) and (ii) hold.
Then for an eigenvalue $\lambda$ of $T$, there exist $n_0$ and a sequence of eigenvalues $\lambda_n$ of $T_n$, $n > n_0$, such that
$\lambda_n \to \lambda$ as $n \to \infty$.   
For any sequence $\lambda_n \in \sigma(T_n)$ with this convergence property and the associated eigenfunctions $x_n \in {\mathcal N}(F_n(\lambda_n))$,
$\|x_n\|=1$, it holds that 
\begin{eqnarray}
	\label{eigenvalueorder}	|\lambda_n - \lambda| &\le& C \epsilon_n^{1/\kappa}, \\
	\label{eigenvectororder}	\inf_{x \in G(\lambda)} \| x_n - p_n x\| &\le& C \epsilon_n^{1/\kappa},
	\end{eqnarray}
	where $\epsilon_n =  \left\| (T_np_n -p_nT) \right \|$ and $\kappa$ is the ascent of $\lambda$.
\end{theorem}
\begin{proof}
Let $\Omega \subset \mathbb C$ be compact such that $\lambda \in \Omega$ and $\lambda \notin \partial \Omega$. 
Since $T$ is compact, for $F(\cdot)$ defined in \eqref{Feta}, $\rho(F) \ne \emptyset$. 

We check the conditions (b1)-(b4).
Since $0 \notin \Omega$, $F(\eta), \eta \in \Omega,$ is holomorphic Fredholm operator of index zero and thus (b1) holds.
Since $\Omega$ is a compact set, $\frac{1}{\eta} I_n$ is equibounded. The operators $T_n$ is uniformly bounded since $T$ is bounded and (ii) is satisfied (see Theorem (2.55) in \cite{Vainikko1976}).  
Thus $\{F_n(\cdot)\}$ is equibounded, i.e., (b2) holds. 
For $\eta \in \Omega$ and $x \in X$, (b3) holds since 
\begin{eqnarray*}
\|[F_n(\eta)p_n-p_n F(\eta)] x\| &=& \left\| \left[ \left(T_n - \frac{1}{\eta} I_n\right)p_n - p_n \left(T-\frac{1}{\eta}I \right) \right] x \right\| \\
&=&  \left\| \left[ \left(T_np_n -p_nT\right) -  \left( \frac{1}{\eta} I_np_n-\frac{1}{\eta}p_nI \right) \right] x \right\| \\
&\le& \left\| (T_np_n -p_nT)x \right \|  \\
&\to& 0
\end{eqnarray*}
due to conditions (i) and (ii) and the fact that $(I_np_n -p_nI)x =0$. The regular convergence of $F_n(\eta)$ to $F$ follows Theorem~\ref{FnFregular} and thus (b4) holds.

Finally, we have that 
\[
\max_{|\eta-\lambda| \le \delta}\max_{x \in G(\lambda)} \| F_n(\eta)p_n x - p_n F(\eta) x\| = \max_{x \in G(\lambda)} \left\| (T_np_n -p_nT)x \right \| \le \left\| T_np_n -p_nT \right \|.
\]
The proof is complete by applying Theorem~\ref{Thm210}.
\end{proof}

\begin{remark}
In view of Remark~\ref{Tnprime}, if $X_n \subset X$, one has that
\begin{equation}\label{consistencyErr}
\| (T_n' -p_nT)x \| \le \|(T_n'-T) x\| + \|T_n x - p_n T_n x\|.
\end{equation}
Hence the consistency error $\epsilon_n$ is bounded by the sum of the error of the finite element solution and the approximation error of the discrete space.
\end{remark}

The above theorem (c.f. Theorem \ref{Thm210}) claims that all eigenvalues and eigenfunctions are approximated correctly. 
Note that condition (b3) (point-wise convergence) is not sufficient to rule out spurious discrete eigenvalues. We refer the readers to \cite{Boffi2000MC} for some
mixed finite element methods that produce spurious eigenvalues.

We end this section with the following lemma which can be useful to study the convergence for the eigenvalue problem in a space different than $L^2(D)$.
\begin{lemma} Let $X$ and $Y$ be Banach spaces such that $Y \subset X$. 
Assume that the embedding of $Y$ into $X$ is compact. Let $T: X \to Y \subset X$ be a bounded linear operator. 
Then the restriction of $T$ on $Y$, $T|_Y: Y \to Y$, is compact.
\end{lemma}
\begin{proof} Let $\{x_n\}, n \in \mathbb N$ be a bounded sequence in $Y$. Due to the compact embedding of $Y$ into $X$, there exists a convergent subsequence
$\{x_{n'}\}, n' \in N' \subset \mathbb N,$ in $X$. Let $x = \lim_{n' \to \infty} x_{n'}$ and $y = Tx  \in Y$. Clearly, $\lim_{n' \to \infty} Tx_{n'} = y$, i.e.,
$\{Tx_{n'}\}$ converges to $y \in Y$. The proof is complete. 
\end{proof}

\section{Dirichlet Eigenvalue Problem}
In this section, we analyze the convergence of several finite element methods for the Dirichlet eigenvalue problem. 
Let $D\subset \mathbb{R}^2$ be a bounded Lipschitz polygonal domain.
The Dirichlet eigenvalue problem is to find $\lambda\in \mathbb R$ and $u\neq 0$ such that
\begin{subequations}\label{possioneig}
\begin{align}
\label{possionEA}-\Delta u=\lambda u \qquad &\text{in } D,\\[1mm]
\label{possionEB} u = 0 \qquad &\text{on } \partial D.
\end{align}
\end{subequations}
The associated source problem is, given $f$, to find $u$ such that
\begin{subequations}\label{sourcepro}
\begin{align}
\label{possionEA}-\Delta u=f \qquad &\text{in } D,\\[1mm]
\label{possionEB} u = 0 \qquad &\text{on } \partial D.
\end{align}
\end{subequations}
%

For $f\in L^2(D)$, the weak formulation of \eqref{sourcepro} is to find $u\in H^1_0(D)$ such that
\begin{equation}\label{vsource}
a(u,v)=(f,v)\quad \text{for all }~v\in H^1_0(D),
\end{equation}
where
\[
a(u,v)=\int_D \nabla u\cdot\nabla \bar{v}~dx,\quad (f,v)=\int_D f \bar{v}~dx.
\]
The variational formulation for the eigenvalue problem is to find $\lambda \in \mathbb R$ and $u \in H_0^1(D)$ such that
\begin{equation}\label{veig}
a(u,v)=\lambda(u,v)\qquad \text{for all }~v\in H^1_0(D),
\end{equation}

There exists a unique solution $u \in H_0^1(D)$ to \eqref{vsource}. Furthermore, $u \in H^{1+\alpha}(D)$, where 
 $1/2 < \alpha \le 1$ ($\alpha = 1$ if $D$ is convex) is the elliptic regularity index (see, e.g., Sec. 3.2 in \cite{SunZhou2016}).
Due to the wellposedness of \eqref{vsource} and the compact imbedding of $H^1_0(D)$ into $L^2(D)$, 
there exists a compact solution operator to \eqref{vsource}
\begin{equation}
T: L^2(D)\rightarrow  L^2(D) \quad \text{such that} \quad Tf=u.
\end{equation}
Assuming that $\lambda \ne 0$, the Dirichlet eigenvalue problem is equivalent to the operator eigenvalue problem of finding $\lambda \in \mathbb R$ and $u \in L^2(D)$ such that
\begin{equation}\label{Tlambdau}
T(\lambda u)=u.
\end{equation}

Define a nonlinear operator function $F: \Omega \to \mathcal{L}(L^2(D), L^2(D))$ by
\begin{equation}\label{Flambda}
F(\eta):=T-\frac{1}{\eta}I, \quad \eta \in \Omega.
\end{equation}
Clearly, $\lambda$ is a Dirichlet eigenvalue if and only if $\lambda$ is an eigenvalue of $F(\cdot)$.
\begin{lemma} Let $\Omega\subset \mathbb{C}\backslash\{0\}$ be a compact set. Then
$F(\cdot):\Omega\rightarrow \mathcal{L}(L^2(D),L^2(D))$ is a holomorphic Fredholm operator function of index zero.
\end{lemma}
\begin{proof}
 It is clear that $F(\cdot)$ is holomorphic in $\Omega$. 
 Since $T$ is compact and $I$ is the identity operator, $F(\lambda)$ is a Fredholm operator of index zero.
\end{proof}

Let $X_n$ be a finite element space associated with $\mathcal{T}_{h_n}$ endowed with the $L^2$-norm $\| \cdot\|$.
Let $p_n$ be the $L^2$-projection from $L^2(D)$ to $X_n$ such that $\|x_n-x\|\rightarrow 0$ as $n \rightarrow \infty$ for $x \in X$,
where $x_n = p_n x$. For example, if $X_n$ is the linear Lagrange element space, it holds that 
\begin{equation}\label{pnorder}
\|p_n x - x\| \le C h_n^r \|x\|_{H^r(D)}, \quad 0 \le r \le 2.
\end{equation}

Assume that there exists a finite element solution operator $T_n : X_n \rightarrow X_n$ such that $T_n f_n = u_n$. In general, one has the convergence of the finite element method for the source problem, i.e., 
\begin{equation}\label{Tnpnf}
\lim_{n \to \infty} \|T_n p_n f - p_n T f \| = 0 \quad \text{for all } f \in X.
\end{equation}

Next we investigate several finite element methods for the Dirichlet eigenvalue problem using the results in Section~\ref{FemConv}.

\subsection{Conforming Finite Element Method}
Let $X_n$ be the Lagrange element space equipped with the $L^2$-norm $\| \cdot\|_{X_n} = \| \cdot\|$. 
One has that $X_n \subset H_0^1(D) \subset X$.
The discrete formulation for the source problem \eqref{vsource} is as follows. For $f \in L^2(D)$, find $u_n\in X_n$ such that,
\begin{equation}\label{FemPro}
a(u_n, v_n)=(f_n,v_n) \quad \text{for all } v_n \in X_n,
\end{equation}
where $a(u_n, v_n):=(\nabla u_n, \nabla v_n)$ and $f_n = p_n f$.

The discrete problem \eqref{FemPro} has a unique solution $u_n$ such that $\|u_n\|\leq C\|f_n\|$.
Let $u$ and $u_n$ be the solutions of \eqref{vsource} and \eqref{FemPro}, respectively.
The classical finite element error analysis gives that (see, e.g., \cite{BrennerScott2008})
\begin{equation}\label{DEL2error}
\|u-u_n\|\leq Ch_n^{2\alpha}\|f\|.
\end{equation}
Let $T_n : X_n \rightarrow X_n$, $T_n f_n = u_n$ be the finite element solution operator of \eqref{FemPro}. 
Since 
\[
\|u_n - p_nu\| \le \|u_n - u\|+\|u-p_nu\|,
\] 
using \eqref{DEL2error}, one obtains that
\begin{equation}\label{DECTn}
\lim_{n \to \infty} \|T_n p_n f - p_n T f \| = 0 \quad \text{for all } f \in X.
\end{equation}

Define a discrete operator function $F_n: \Omega \to \mathcal{L}(X_n, X_n)$
\begin{equation}\label{Fhlambda}
F_n(\eta):=T_n-\frac{1}{\eta}I_n.
\end{equation}
\begin{lemma}
$F_n \to F$ regularly.
\end{lemma}
\begin{proof}
Let $X'=H_0^1(D)$. Since $T_n: X_n \to X', n \in \mathbb N$ are uniformly bounded, due to the compact embedding of $X'$ into $X$ and Lemma~\ref{CompactConv},
$T_n \to T$ compactly. Lemma~\ref{SCR} implies that $F_n \to F$ regularly. The proof is complete.
\end{proof}
The convergence for the Lagrange finite element method of finding $\lambda_n$ and $u_n$ such that
\[
a(u_n, v_n)= \lambda_n (u_n,v_n) \quad \text{for all } v_n \in X_n,
\]
follows from \eqref{pnorder}, \eqref{DEL2error}, and Theorem~\ref{FEMconvergence}.
\begin{theorem}\label{CDE} 
Let $\lambda \in \sigma(F)$. There exists $n_0 \in {\mathbb N}$ and a sequence $\lambda_n \in \sigma(F_n)$,
	$n \ge n_0$, such that $\lambda_n \to \lambda$ as $n \to \infty$. For any sequence $\lambda_n \in \sigma(F_n)$
	with this convergence property and the associated eigenfunction $u_n, \|u_n\|=1$, one has that
	\begin{equation}\label{DEConfErr}
		|\lambda_n - \lambda| \le C h_n^{2\alpha} \quad \text{and} \quad \|u_n-u\| \le Ch_n^{2\alpha},
	\end{equation}
	where $u$ is some eigenfunction associated to $\lambda$ with $\|u\|=1$.
\end{theorem}
\begin{proof}
Let $\Omega \subset \mathbb C$ be a simply connected set such that $0 \notin \Omega$ and $\lambda \in \Omega$. 
It is clear that $\rho(F) \cap \Omega \ne \emptyset$.
For $X=Y=L^2(D)$, $X_n =Y_n$ and $p_n=q_n$ being the $L^2$-projection from $X$ onto $X_n$, condition (i) in Theorem~\ref{FEMconvergence} holds. Due to \eqref{DECTn}, (ii) of Theorem~\ref{FEMconvergence} holds. Since $T$ is self-adjoint,
each eigenvalue has ascent $\kappa = 1$. From \eqref{DEL2error}, it holds that $\epsilon_n \le Ch_n^{2\alpha}$. Hence \eqref{DEConfErr} holds due to Theorem~\ref{FEMconvergence}.
\end{proof}
\begin{remark}
For convex domains, one has that $\alpha = 1$ and obtains second order of convergence for the eigenvalues and eigenfunctions using the linear Lagrange element.
\end{remark}
\begin{remark}
We refer the readers to \cite{XiaoEtal2020AML} which uses the holomorphic operator approach but a different proof of regular convergence.
\end{remark}

\subsection{Interior Penalty Discontinuous Galerkin Methods}
We consider the interior penalty discontinuous Galerkin methods for the Dirichlet eigenvalue problem. Following the notations in \cite{Antonietti2006CMAME}, 
let $h_K$ be the diameter of the element $K \in {\mathcal T}_n$. Denote by ${\mathcal F}_n^I$ and ${\mathcal F}_n^B$ be the sets of the interior edges and
boundary edges of ${\mathcal T}_n$, respectively. Let ${\mathcal F}_n := {\mathcal F}_n^I \cup {\mathcal F}_n^B$.
Let ${\boldsymbol w}$ and $v$ be piecewise smooth vector-valued and scalar-valued functions. Let $F \in {\mathcal F}_n^I$ be an interior face share by two elements $K^+$ and $K^-$
with outward norm ${\boldsymbol \nu}^\pm$. Denote by ${\boldsymbol w}^\pm$ and $v^\pm$ on $\partial K^\pm$ taken from within $K^\pm$, respectively. The jumps
across $F$ are defined by
\[
\llbracket {\boldsymbol w} \rrbracket = {\boldsymbol w}^+ \cdot {\boldsymbol \nu}^+ + {\boldsymbol w}^- \cdot {\boldsymbol \nu}^-, \quad
\llbracket v \rrbracket = v^+ {\boldsymbol \nu}^+ + v^- {\boldsymbol \nu}^-
\]
and the averages are defined by
\[
\mv{{\boldsymbol w}} = \frac{1}{2} \left( {\boldsymbol w}^+ +{\boldsymbol w}^-\right), \quad \mv{v} = \frac{1}{2} \left(v^+ +v^-\right).
\]
For $F \in {\mathcal F}_n^B$, one simply defines
\[
\llbracket {\boldsymbol w} \rrbracket = {\boldsymbol w} \cdot {\boldsymbol \nu}, \quad \llbracket v \rrbracket = v {\boldsymbol \nu}, \quad
\mv{{\boldsymbol w}} ={\boldsymbol w}, \quad \mv{v} = v.
\]
Define the discontinuous Galerkin space $X_n$ by
\begin{equation}\label{DGSpace}
	X_n:=\{ v \in L^2(D) : v|_{K} \in P^\ell(K), K \in {\mathcal T}_n \},
\end{equation}
where $P^\ell(K)$ is the space of polynomials of degree at most $\ell \ge 1$ on $K$. 

Let $p_n$ be the $L^2$-projection of $f \in L^2(D)$ to $X_n$ such that
\begin{equation}\label{L2projection}
(p_n f, v_n) = (f, v_n) \quad \text{for all } v_n \in X_n.
\end{equation} 
Then condition (i) in Theorem~\ref{FEMconvergence} holds. 

We consider the  symmetric interior penalty (SIP) DG methods in primal form for the Poisson equation. Find $u_n \in X_n$ such that
\begin{equation}\label{DGmethods}
a_n(u_n, v_n) = (f_n, v_n) \quad \text{for all } v_n \in X_n
\end{equation}
where $a_n: X_n \times X_n \to \mathbb C$ is defined by
\begin{eqnarray*}
a_n(u_n, v_n) &=& (\nabla_n u_n, \nabla_n v_n) - \int_{{\mathcal F}_n} \mv{\nabla_n u_n} \cdot  \llbracket \bar{v}_n \rrbracket ds \\
&& \quad - \int_{{\mathcal F}_n}  \mv{\nabla_n \bar{v}_n} \cdot  \llbracket u_n \rrbracket ds - s_n(u_n, v_n) \quad \text{for } u_n, v_n \in X_n.
\end{eqnarray*}
The interior penalty family are defined by choosing the the stabilization form $s_n(\cdot, \cdot)$ as
\[
s_n(u_n, v_n) = \int_{{\mathcal F}_h} \gamma h_n^{-1}  \llbracket u_n \rrbracket  \llbracket \bar{v}_n \rrbracket ds,
\]
for $\gamma > 0$ independent of the mesh size. 

The finite element space $X_n$ can be endowed with a second norm $\| \cdot\|_{n}$:
\begin{equation}\label{VnNorm}
\|v_n\|_{n} = \|\nabla_n v_n\|+\|h^{-1/2} \llbracket v_n  \rrbracket \|^2_{{\mathcal F}_n},
\end{equation}
where $\nabla_n$ is the element-wise gradient operator. For $f \in H_0^1(\Omega)$ or $f \in X_n$, the Poincar\'{e} inequality holds (Property 1 in \cite{Antonietti2006CMAME})
\begin{equation}\label{Poincare}
\|f\| \le C \|f\|_{n},
\end{equation}
where $C$ is a constant depends on $D$ but not on the mesh.


For $\gamma$ large enough, it is well-known that there exists a unique solution $u_n$ to \eqref{DGmethods}.
Denote the discrete solution operator by $T_n: X_n \to X_n$.
According to Property 2 in \cite{Antonietti2006CMAME}, for the discrete solution $u_n$ to \eqref{DGmethods} and the exact solution $u$ to \eqref{vsource}, it holds that
\begin{equation}\label{DGerror}
\|u-u_n\|_{n} \le C h^\alpha \|f\| \quad \text{for } f \in L^2(D).
\end{equation}
Using the approximation property of $p_n$ and the Poincar\'{e} inequality \eqref{Poincare}, for $f \in L^2(D)$, one has that
\begin{eqnarray}
\nonumber \|T_n p_n f - p_n T f\| &=& \|u_n - p_nu\| \\
\nonumber &\le&\|u_n - u\| + \|u-p_n u\| \\
\nonumber &\le&\|u_n - u\|_{n} + \|u-p_n u\| \\
\nonumber &\le& C h^\alpha \|f\| + C h^\alpha \|u\|_{H^1(D)} \\
\label{DGorder} &\le& C h^\alpha \|f\|.
\end{eqnarray}
Similarly, we define the operator function $F_n: \Omega \to \mathcal{L}(X_n, X_n)$ such that
\[
F_n(\eta) : = T_n - \frac{1}{\eta}I_n,
\]
where $X_n$ is the discontinuous Galerkin space defined in \eqref{DGSpace}, $T_n$ is the discrete solution operator to \eqref{DGmethods},
and $I_n$ is the identity operator from $X_n$ to $X_n$.
\begin{theorem}\label{DGDirichlet} 
Let $\lambda \in \sigma(F)$. There exists $n_0 \in {\mathbb N}$ and a sequence $\lambda_n \in \sigma(F_n)$,
	$n \ge n_0$, such that $\lambda_n \to \lambda$ as $n \to \infty$. For any sequence $\lambda_n \in \sigma(F_n)$
	with this convergence property and the associated eigenfunction $u_n, \|u_n\|=1$, one has that
	\begin{equation}\label{DEConfErr}
		|\lambda_n - \lambda| \le C h_n^{\alpha} \quad \text{and} \quad \|u_n-u\| \le Ch_n^{\alpha},
	\end{equation}
	where $u$ is some eigenfunction associated to $\lambda$ with $\|u\|=1$.
\end{theorem}
\begin{proof}
From \eqref{DGorder}, $T_n \to T$ uniformly. 
Thus $F_n(\eta)$ converges to $F(\eta)$ regularly due to Lemma~\ref{U2C}. Then \eqref{DEConfErr} follows Theorem~\ref{FEMconvergence}.
\end{proof}
We refer the readers to \cite{XiJi2021} which uses the holomorphic operator approach but a different discrete norm for $X_n$ and a different proof.
\begin{remark}
Note that the convergence is not optimal since \eqref{DGorder} is obtained using the Poincar\'{e} inequality and the convergence of the DG solution in $\|\cdot\|_{X_n}$ norm.
If a higher order of convergence of $u_n$ in $L^2$ norm is available, the convergence orders in Theorem~\eqref{DGDirichlet} can be improved.
\end{remark}

\subsection{Nonconforming Crouzeix-Raviart Method}
We consider the nonconforming piecewise linear finite element space of Crouzeix-Raviart \cite{Braess2001}:
\begin{eqnarray}
\nonumber &&X_n :=\{v: v|_K \in {\mathcal P}_1 \text{ is continuous at the midpoints of the edges of } K\\
\label{CRSpace} && \qquad \qquad \qquad \text{ and }  v = 0 \text{ at the midpoints on } \partial D\},
\end{eqnarray}
where ${\mathcal P}_1$ denotes the space of polynomials of degree less or equal to $1$.

Define the bilinear form on $X_n$
\[
a_n(u_n, v_n):=\sum_{K \in {\mathcal T}_n} \int_K \nabla u_n \cdot \nabla v_n dK, \quad u_n, v_n \in X_n.
\]
The discrete problem is to find $u_n \in X_n$ such that
\begin{equation}\label{CRsource}
a_n(u_n, v_n) = (f_n, v_n) \quad \text{for all } v_n \in X_n,
\end{equation}
where $f_n = p_n f$.
There exists a unique solution $u_n$ to \eqref{CRsource}. Denote the discrete solution operator to be $T_n: X_n \to X_n$ such that
\[
u_n = T_n f_n.
\] 
Using the error estimate in $L^2$-norm (Theorem 1.5 of Chp. III in \cite{Braess2001}) and the property of the $L^2$-projection, one has that
\begin{equation}\label{CRconvergence}
\| p_nTf - T_n p_n f\| \le C h^{2\alpha} \|f\|,
\end{equation}
which implies $T_n \to T$ uniformly.

Define the discrete operator function $F_n: \Omega \to \mathcal{L}(X_n, X_n)$ such that
\[
F_n(\eta) : = T_n - \frac{1}{\eta}I_n,
\]
where $X_n$ is the non-conforming Crouzeix-Raviart finite element space defined in \eqref{CRSpace}, $T_n$ is the discrete solution operator to \eqref{CRsource},
and $I_n$ is the identity operator from $X_n$ to $X_n$. Using Theorem~\ref{FEMconvergence}, 
we obtain the convergence of the non-conforming Crouzeix-Raviart method.
It is proof is similar to the previous ones and thus omitted. 
\begin{theorem}\label{CRDirichlet} 
Let $\lambda \in \sigma(F)$. There exists $n_0 \in {\mathbb N}$ and a sequence $\lambda_n \in \sigma(F_n)$,
	$n \ge n_0$, such that $\lambda_n \to \lambda$ as $n \to \infty$. For any sequence $\lambda_n \in \sigma(F_n)$
	with this convergence property and the associated eigenfunction $u_n, \|u_n\|=1$, one has that
	\begin{equation}\label{ConfErr}
		|\lambda_n - \lambda| \le C h_n^{2\alpha} \quad \text{and} \quad \|u_n-u\| \le Ch_n^{2\alpha},
	\end{equation}
	where $u$ is some eigenfunction associated to $\lambda$ with $\|u\|=1$.
\end{theorem}

\begin{remark}
The Crouzeix-Raviart element is non-conforming in the sense that it is not a subspace of $H^1_0(D)$, which is the solution space for \eqref{vsource}. 
Note that we use $L^2(D)$ other than $H^1(D)$. Similar comment applies to the Morley element method for the biharmonic eigenvalue problem in the next section.
\end{remark}


\section{The Biharmonic Eigenvalue Problem}
We now consider the biharmonic eigenvalue problem. 
Let $D$ denote a bounded Lipschitz polygonal domain in $\mathbb{R}^2$ with boundary $\partial D$. Let $\nu$ denote the unit outward normal to $\partial D$. 
The biharmonic eigenvalue problem with clamped plate boundary condition is to find $\lambda \in \mathbb R$ and $u \ne 0$ such that 
\begin{subequations}\label{biharmonicE}
\begin{align}
\label{biharmonicEA}\Delta^2 u=\lambda u \qquad &\text{in } D,\\[1mm]
\label{biharmonicEB}u=\frac{\partial u}{\partial \nu} = 0 \qquad &\text{on } \partial D.
\end{align}
\end{subequations}
The associated source problem is as follows. Given a function $f$, find a function $u$ such that
\begin{subequations}\label{biharmonicS}
\begin{align}
\label{biharmonicSA}\Delta^2 u=f \qquad&\text{in } D,\\[1mm]
\label{biharmonicSB}u=\frac{\partial u}{\partial \nu} = 0 \qquad &\text{on } \partial D.
\end{align}
\end{subequations}

Define
\[
H_0^2(D) :=\left\{ v \in H^2(D): v = \frac{\partial v}{\partial \nu}=0 \text{ on } \partial D \right\}, 
\]
and a sesquilinear form $a: H_0^2(D) \times H_0^2(D) $ such that
\begin{equation}\label{auvTuTv}
a(u, v):= (\triangle u, \Delta v).
\end{equation}
The weak formulation for \eqref{biharmonicS} is, for $f \in L^2(D)$, to find $u \in H_0^2(D)$ such that
\begin{equation}\label{biharmonicSweakC}
a(u,v) =( f,v) \quad \text{for all } v \in H_0^2(D) .
\end{equation}
The weak formulation for \eqref{biharmonicE} is to find $\lambda \in \mathbb{R}$ and $u \in H_0^2(D), u\not=0$ such that
\begin{equation}\label{biharmonicEweak}
a(u,v) =\lambda(u,v) \quad  \text{for all } v \in H_0^2(D) .
\end{equation}

There exists a unique solution $u$ to \eqref{biharmonicEweak} belonging to $H^{2+\alpha}(D)$ 
for some $\alpha \in (1/2, 1]$ such that
\begin{equation}\label{BiharmonicReg}
\|u\|_{H^{2+\alpha}(D)} \le C \|f\|,
\end{equation}
where the constant $C$ depending only on $D$. When $D$ is convex, $\alpha = 1$. The parameter $\alpha$ is referred as 
the \index{index of elliptic regularity}index of elliptic regularity for the biharmonic equation.

Consequently, there exists a solution operator $T: L^2(D) \to L^2(D)$ such that, given $f \in L^2(D)$,
\begin{equation}\label{biharmonicT}
a(Tf, v) = (f, v) \quad \text{for all } v \in H_0^2(D).
\end{equation}
It is obvious that $T$ is self-adjoint due to the symmetry of $a(\cdot, \cdot)$ and compact due to the compact imbedding of $H^2_0(D)$ into $L^2(D)$.
Similar to the case of the Dirichlet eigenvalue problem, we define a holomorphic Fredholm operator function $F: \Omega \to \mathcal{L}(L^2(D), L^2(D))$ such that
\begin{equation}\label{BiFlambda}
F(\eta):=T-\frac{1}{\eta}I, \quad \eta \in \Omega.
\end{equation}

\subsection{Argyris Element Method}
We consider the Argyris element (see, e.g., Section 4.2 of \cite{SunZhou2016}), which is $H^2$-conforming for triangular meshes $\mathcal{T}_{h_n}$.
Denote the associated finite element space by $X_n$. 

The discrete problem for the source problem \eqref{biharmonicS} can be stated as follows.
For $f \in L^{2}(D)$, find $u_n \in X_n \subset H_0^2(D)$ such that
\begin{equation}\label{biharmonicSDweak}
a(u_n,v_n) =(p_n f,v_n) \quad \text{for all } v_n \in X_n.
\end{equation}
There exists a unique solution $u_n$ to \eqref{biharmonicSDweak} such that 
\begin{equation}\label{uhhH2fH2}
\|u-u_n\|_{H^2(D)} \le C h_n^{\alpha} \|f\|.
\end{equation}
The discrete formulation for the eigenvalue problem \eqref{biharmonicE} is to find $\lambda_n \in \mathbb{R}$ and $u_n \in X_n$, $u_n\not=0$ such that
\begin{equation}\label{biharmonicEDweak}
a(u_n, v_n) =\lambda_n (u_n, v_n) \quad \text{for all } v_n \in X_n.
\end{equation}

Using a duality argument (Chp. II of \cite{Braess2001}) and \eqref{uhhH2fH2}, one has that
\begin{equation}\label{biharmonicL2}
\|u-u_n\| \le C h_n^{2\alpha} \|f\|.
\end{equation}
Consequently, the discrete solution operator $T_n: X_n \to X_n$ is such that
\begin{equation}\label{BiCTn}
\|T_n p_n f - p_n T f \| \le C  h_n^{2\alpha} \|f\| \quad \text{for } f \in L^2(D).
\end{equation}
Since $X_n \subset H_0^2(D), n \in \mathbb N$ and the embedding of $H_0^2(D)$ into $X$ is compact, one has that $T_n \to T$ compactly due to Lemma~\ref{CompactConv}.

Define a discrete operator function $F_n: \Omega \to \mathcal{L}(X_n, X_n)$
\begin{equation}\label{BiFhlambda}
F_n(\eta):=T_n-\frac{1}{\eta}I_n.
\end{equation}
Then $F_n \to F$ regularly.
Using a similar argument as for the Dirichlet eigenvalue problem, we obtain the following convergence theorem of the Argyris element method for the biharmonic eigenvalue problem. 
\begin{theorem}\label{BiArgyris} 
Let $\lambda \in \sigma(F)$. There exists $n_0 \in {\mathbb N}$ and a sequence $\lambda_n \in \sigma(F_n)$,
	$n \ge n_0$, such that $\lambda_n \to \lambda$ as $n \to \infty$. For any sequence $\lambda_n \in \sigma(F_n)$
	with this convergence property and the associated eigenfunction $u_n, \|u_n\|=1$, one has that
	\begin{equation}\label{ConfErr}
		|\lambda_n - \lambda| \le C h_n^{2\alpha} \quad \text{and} \quad \|u_n-u\| \le Ch_n^{2\alpha},
	\end{equation}
	where $u$ is some eigenfunction associated to $\lambda$ with $\|u\|=1$.
\end{theorem}

\subsection{$C^0$ interior penalty discontinuous Galerkin method}
We consider the $C^0$ interior penalty discontinuous Galerkin method ($C^0$ IPG) for the biharmonic equation \cite{BrennerMonkSun2015} 
(see also \cite{XiJi2022JSC} using the holomorphic operator approach but a different proof).
Let $X_n \subset H^1(\Omega)$ be the Lagrange finite
 element space of order $k \ge 2$ associated with $\mathcal{T}_{h_n}$.
 Let $\mathcal{E}_{h_n}$ be the set of the edges in $\mathcal{T}_{h_n}$. For edges $e \in \mathcal{E}_{h_n}$ that are
 the common edge of two adjacent triangles $K_{\pm} \in \mathcal{T}_{h_n}$ and for
 $v\in X_n$, we define the jump of the flux to be
\[
 \llbracket \partial v/\partial n_e \rrbracket
  = \frac{\partial v_{\scriptscriptstyle K_+}}{\partial n_e} \Big{|}_e -
 \frac{\partial v_{\scriptscriptstyle K_-}}{\partial n_e} \Big{|}_e,
\]
  where $n_e$ is the unit normal pointing from $K_-$ to $K_+$.
 Let
\[
 \frac{\partial^2 v}{\partial n_e^2} = n_e \cdot (\triangle v) n_e
\]
 and define the average normal-normal derivative to be
\[
 \mv{\frac{\partial^2 v}{\partial n_e^2}} = \frac{1}{2}
  \left( \frac{\partial^2 v_{\scriptscriptstyle K_+}}{\partial n_e^2}
  +\frac{\partial^2 v_{\scriptscriptstyle K_-}}{\partial n_e^2}\right).
\]
 For $e \in \partial D$, we take $n_e$ to be the unit outward normal and define
\[
\llbracket\partial v/\partial n_e \rrbracket = -\frac{\partial v}{\partial n_e}
 \quad\text{and}
 \quad  \mv{\frac{\partial^2 v}{\partial n_e^2}}=\frac{\partial^2 v}{\partial n_e^2}.
\]

Given $f_n \in X_n$, $f_n = p_n f$, the corresponding $C^0$ IPG method for the source problem
 is to find $u_h\in X_n$ such that
\begin{equation}\label{DiscreteSourceProblem}
  a_n(u_n, v_n)=(f_n, v_n) \quad\text{for all }\, v\in X_n,
\end{equation}
where
\begin{align}
 a_n(w,v) &=   \sum_{K \in {\mathcal T}_{h_n}} \int_K D^2 w : D^2 v  \,\text{d}x \nonumber\\
 & \hspace{30pt} +
 \sum_{e \in {\mathcal{E}}_{h_n}}  \int_e \mv{
					\frac{\partial^2 w}{\partial n_e^2}}
  \left \llbracket \frac{\partial v}{\partial n_e} \right\rrbracket
 +  \mv{\frac{\partial^2 v}{\partial n_e^2} }
  \left \llbracket\frac{\partial w}{\partial n_e} \right\rrbracket \,\text{d}s \nonumber\\
 & \hspace{30pt} +   \sigma\sum_{e \in  {\mathcal{E}}_{h_n}} \frac{1}{|e|} \int_e
 \left \llbracket \frac{\partial w}{\partial n_e} \right\rrbracket
		   \left \llbracket\frac{\partial v}{\partial n_e} \right\rrbracket \,\text{d}s \label{ahCP},
\end{align}
where $D^2 w : D^2 v = \sum_{i,j =1}^2 w_{x_i x_j} v_{x_i x_j}$ is the Frobenius inner product of the Hessian
matrices of $w$ and $v$, and $\sigma>0$ is a (sufficiently large) penalty parameter.

There exists discrete solution operators $T_n: X_n \to X_n$ to \eqref{DiscreteSourceProblem} such that $T_n \to T$.

 The $C^0$ IPG method for the biharmonic eigenvalue problem is to find
 $\lambda_n \in \mathbb R$ and $u_n\not=0$ such that
\begin{equation}\label{C0IPEigProblems}
  a_n(u_n, v_n)=\lambda_n (u_n, v_n) \quad\text{for all }\, v_n\in X_n.
\end{equation}

Consequently, define
\[
F_n(\eta) = T_n - \frac{1}{\eta} I_n.
\] 
Let $X'=H_0^1(D)$. Since $T_n: X_n \to X', n \in \mathbb N$ are uniformly bounded, due to the compact embedding of $X'$ into $X$ and Lemma~\ref{CompactConv},
$T_n \to T$ compactly. Lemma~\ref{SCR} implies that $F_n \to F$ regularly.
Then the following convergence theorem holds for the $C^0$ IPG method.
\begin{theorem}\label{BiC0IPG} 
Let $\lambda \in \sigma(F)$. There exists $n_0 \in {\mathbb N}$ and a sequence $\lambda_n \in \sigma(F_n)$,
	$n \ge n_0$, such that $\lambda_n \to \lambda$ as $n \to \infty$. For any sequence $\lambda_n \in \sigma(F_n)$
	with this convergence property and the associated eigenfunction $u_n, \|u_n\|=1$, one has that
	\begin{equation}\label{ConfErr}
		|\lambda_n - \lambda| \le C h_n^{2\alpha} \quad \text{and} \quad \|u_n-u\| \le Ch_n^{2\alpha},
	\end{equation}
	where $u$ is some eigenfunction associated to $\lambda$ with $\|u\|=1$.
\end{theorem}
\begin{proof}
Due to Lemma 1 in \cite{BrennerMonkSun2015}, there exists a unique discrete solution $u_n$ to \eqref{C0IPEigProblems}, such that
\begin{equation}
\|u-u_n \| \le Ch_n^{2\alpha} \|f\| \quad \text{for } f \in L^2(D).
\end{equation}
One has that
\begin{equation}\label{BiC0IPGTn}
\|T_n p_n f - p_n T f \| \le C  h_n^{2\alpha} \|f\| \quad \text{for } f \in L^2(D).
\end{equation}
\end{proof}

\subsection{Morley Element Method}
We consider the Morley element space $X_n$. Let $p_n$ be the $L^2$-projection from $X$ onto $X_n$.
%
For $f \in X$, the Morley finite element method for the biharmonic equation is to find $ u_n \in X_n$ such that 
\begin{equation}\label{BiharmonicMorley}
a_n(u_n,  v_n )=  (p_n f, v_n) \quad\text{for all }\,v_n\in X_n,
\end{equation}
where 
\[
a_n(u_n, v_n):=\sum_{K \in {\mathcal T}_n} \int_K  D^2 u_n : D^2 v_n dK, \quad u_n, v_n \in X_n.
\]
There exists a unique solution $u_n \in X_n$ such that (see, e.g., Section 4.4.2 of \cite{SunZhou2016})
\begin{equation}\label{ErrorMorley}
\|u-u_n\|_{2,h_n} \le Ch_n^\alpha \|f\|, 
\end{equation}
where the mesh dependent norm $\| \cdot \|_{2,h_n}$ is defined as
\[
 \|u\|_{2, h_n} = \sum_{K \in \mathcal{T}_n} (u, u)_{H^2(K)},
\]

Let $T_n: X_n \to X_n$ be the discrete solution operator to \eqref{BiharmonicMorley}. 
Since $\|u\| \le \|u\|_{2, h_n}$, due to \eqref{ErrorMorley}, it holds that
\[
\|u- u_n\| \le  Ch_n^\alpha \|f\|.
\]
As a consequence, one has that
\begin{equation}\label{BiMorleyTn}
\|T_n p_n f - p_n T f \| \le C  h_n^{\alpha} \|f\| \quad \text{for } f \in X.
\end{equation}

Letting 
\[
F_n(\eta) = T_n - \frac{1}{\eta} I_n, \quad \eta \in \Omega,
\] 
we obtain the following convergence theorem for the Morley element method.
\begin{theorem}\label{BiC0IPG} 
Let $\lambda \in \sigma(F)$. There exists $n_0 \in {\mathbb N}$ and a sequence $\lambda_n \in \sigma(F_n)$,
	$n \ge n_0$, such that $\lambda_n \to \lambda$ as $n \to \infty$. For any sequence $\lambda_n \in \sigma(F_n)$
	with this convergence property and the associated eigenfunction $u_n, \|u_n\|=1$, one has that
	\begin{equation}\label{ConfErr}
		|\lambda_n - \lambda| \le C h_n^{\alpha} \quad \text{and} \quad \|u_n-u\| \le Ch_n^{\alpha},
	\end{equation}
	where $u$ is some eigenfunction associated to $\lambda$ with $\|u\|=1$.
\end{theorem}

\begin{remark}
The convergence order can be improved if a sharper error estimate in $L^2$-norm is available, see, e.g., \cite{Gallistl2015IMANA}.
\end{remark}

\section{Conclusions and Future Work}
In this paper, we present a general approach to prove the convergence of various finite element methods for eigenvalue problems,
including the conforming methods, discontinuous Galerkin methods, and non-conforming methods. Using the abstract approximation theory for
eigenvalue problems of holomorphic Fredholm operator functions, one needs to verify the property of the $L^2$-projection
and the compact convergence of the discrete solution operators in $L^2$-norm. 
The result has the potential to prove the convergence of many other finite elements methods
for eigenvalue problems such as the mixed finite element methods, virtual element methods, and weak Galerkin methods. 

We use $L^2(D)$ and the $L^2$-projection for the finite element spaces. Thus the convergence of the eigenfunctions is also in $L^2$-norm.
However, this framework also works if one chooses other spaces and projections, 
for example, $H^1(D)$ and $H^1$-projection to the finite element spaces for the Dirichlet eigenvalue problem.
Then one can obtain the convergence of the eigenfunctions in $H^1$-norm.

As seen above, the convergence order obtained using Theorem~\ref{FEMconvergence} for certain method is not optimal.
However, if the optimal convergence order for the source problem is available in $L^2$-norm, the optimal convergence order for the eigenfunctions 
follows immediately. There are many interesting projects related to the proposed approach, e.g.,
1)  to prove the convergence of other methods, e.g., the virtual element methods; 2)  to prove the convergence of eigenvectors in other norms, e.g., the maximum norm;
3) to treat other eigenvalue problems, e.g., the Steklov eigenvalue problem; 4) to treat nonlinear eigenvalue problems, e.g., band structure calculation for photonic
crystals of dispersive media and the nonlinear plate vibrations \cite{XiaoEtal2020AML, Xiao2021, Pang2022}.

\bibliographystyle{amsplain}

\end{document}